\documentclass{amsart}

\usepackage{ adjustbox }
\usepackage{ empheq } 
\usepackage{ mathtools }
\usepackage{ amssymb }
\usepackage{ amsmath }
\usepackage{ array }
\usepackage{ booktabs }
\usepackage{ cite }
\usepackage[dvipsnames]{ xcolor }
\usepackage{ enumitem }
\usepackage{ latexsym }
\usepackage{ url }
\usepackage[all]{ xy }
\usepackage{ youngtab }

\theoremstyle{definition}

\newtheorem{definition}{Definition}[section]

\newtheorem{conjecture}[definition]{Conjecture}

\newtheorem{remark}[definition]{Remark}

\theoremstyle{plain}

\newtheorem{theorem}[definition]{Theorem}
\newtheorem{corollary}[definition]{Corollary}

\newcommand{\ff}{\mathcal{LT}}
\newcommand{\pq}{\mathcal{W}}

\allowdisplaybreaks

\begin{document}

\title[POLYNOMIAL IDENTITIES FOR MUTATION ALGEBRAS]
{HIGHER POLYNOMIAL IDENTITIES FOR MUTATIONS OF ASSOCIATIVE ALGEBRAS}

\author{Murray R. Bremner}

\address{Department of Mathematics and Statistics,
University of Saskatchewan,
Saskatoon,
Canada}

\email{bremner@math.usask.ca}

\author{Jose Brox}

\address{University of Coimbra, CMUC,
Department of Mathematics,
3004-504, Coimbra,
 Portugal}

\email{josebrox@mat.uc.pt}

\author{Juana S\'anchez-Ortega}

\address{Department of Mathematics and Applied Mathematics,
University of Cape Town,
South Africa}

\email{juana.sanchez-ortega@uct.ac.za}

\subjclass[2020]{Primary 18M70; Secondary 16R10, 16W10, 17A30, 17A50, 17B60, 17C65, 17D25, 68W30}

\keywords{Associative algebras, mutation algebras, nonassociative algebras,
Lie-admissible, Jordan-admissible,
polynomial identities, algebraic operads,
representation theory of the symmetric group,
computer algebra,
theoretical particle physics.}

\thanks{Murray Bremner was supported by the Discovery Grant \emph{Algebraic Operads}
from NSERC, the Natural Sciences and Engineering Research Council of Canada.
Jose Brox was supported by the Portuguese Government through grant SFRH/BPD/118665/2016 (FCT/Centro 2020/Portugal 2020/ESF).
Juana S\'anchez-Ortega was supported by a CSUR grant from the NRF,
the National Research Foundation of South Africa. This work was
partially supported by the Centre for Mathematics of the University of Coimbra - UID/MAT/00324/2020, funded by the Portuguese Government through
FCT/MCTES.
\includegraphics[width=75pt]{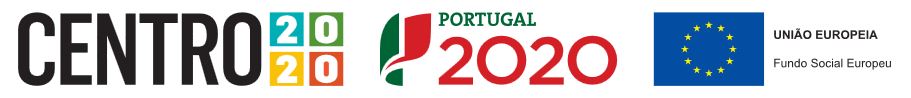}}

\begin{abstract}
We study polynomial identities satisfied by the mutation product $xpy - yqx$
on the underlying vector space of an associative algebra $A$, where $p, q$ are
fixed elements of $A$.
We simplify known results for identities in degree $4$,
proving that only two identities are necessary and sufficient to generate them all; in degree 5,
we show that adding one new identity suffices; in degree 6, we demonstrate the existence of a number of new identities.
\end{abstract}

\maketitle


\section{Introduction}

Let $A$ be an associative algebra over a field $\mathbb{F}$ of characteristic 0.
Fix two elements $p, q \in A$ and define a new bilinear operation
on the underlying vector space:
\[
x \ast_{pq} y = xpy - yqx.
\]
The resulting nonassociative algebra $A_{pq}$ is called the $pq$-\emph{mutation} of $A$.

Mutation algebras were introduced by theoretical physicists around 1980; see
\cite[equation (1.6b)]{KMS} and \cite[equation (66)]{S1981}.
For a survey of early work by mathematicians on this topic, see \cite{MS}.
For a detailed exposition of the structure theory of mutation algebras, see \cite{EMbook}.
For mutations of nonassociative algebras, see \cite{Boers}.

To motivate the investigation of polynomial identities for mutation algebras,
we paraphrase some comments from \cite[Preface]{EMbook}.
Mutation algebras are both Lie- and Jordan-admissible,
but they also satisfy other more complex identities of higher degree;
see \cite[Chapter 5]{EMbook}.
It is an open problem to determine a complete set of independent identities
satisfied by all mutation algebras for arbitrary $p$, $q$.
In fact, mutation algebras do not form a variety defined by polynomial identities.
We note that Lie- and Jordan-admissible algebras were introduced by
Albert \cite[\S IV.1-2]{Albert}.

Polynomial identities for mutation algebras were first studied systematically
by Montaner \cite{M} using the classical techniques of nonassociative algebra
\cite{Osborn,ZSSS}.
That work did not consider the original operation $\ast_{pq}$ but decomposed it
as the sum of commutative and anticommutative operations:
$x \ast_{pq} y = \{x,y\} + [x,y]$ where
\[
\{x,y\} = \tfrac12 ( x \ast_{pq} y + y \ast_{pq} x ),
\qquad
[x,y] = \tfrac12 ( x \ast_{pq} y - y \ast_{pq} x ).
\]
For further information about the notion of polarization of a binary operation,
see \cite{Markl-Remm}.
Furthermore, the work of Montaner considered identities which are not necessarily multilinear, but hand calculation restricted the degree of the identities to $n \le 4$.

We use a different approach which allows us to simplify
the known results in degree 4, to determine a complete set of identities in degree 5, and to demonstrate the existence of a number of new identities in degree 6:
\begin{itemize}[leftmargin=*]
\item
We use elementary concepts from the theory of algebraic operads \cite{BD,LV,Markl,MSS}.
\item
Our main tool is computer algebra, in particular:
\begin{itemize}[leftmargin=*]
\item
Linear algebra over the rational numbers and finite prime fields:
the row canonical form of a matrix using Gaussian elimination.
\item
Linear algebra over the integers:
the Hermite normal form of a matrix and
the Lenstra-Lenstra-Lov\'asz algorithm (LLL, see \cite{LLL,BP}) for lattice basis reduction.
\end{itemize}
\item
We consider only multilinear identities for the original operation $x \ast_{pq} y$:
this allows us to use the representation theory of the symmetric group \cite{BMP}
to decompose the computations into small pieces
corresponding to irreducible representations.
\end{itemize}


\section{Algebraic Operads}


\subsection{The free nonsymmetric operad}

We write $T_n$ for the set of all complete rooted plane binary trees with $n$ leaves
denoted by asterisks; for $n = 1$ there is only the exceptional tree with one leaf and no root,
$T_1 = \{ \ast \}$.
Each tree in $T_n$ contains $n{-}1$ internal nodes (including the root);
hence the size of $T_n$ is the Catalan number $\tfrac{1}{n}\binom{2n{-}2}{n{-}1}$.
We write $U_n$ for the set of all association types in degree $n$:
balanced placements of parentheses in a sequence of $n$ asterisks.
There is a bijection $\mu_n\colon T_n \to U_n$ defined recursively:
$\mu_1(\ast) = \ast$;
for every internal node $v$ with left and right subtrees
$t_1 \in T_{n_1}$ and $t_2 \in T_{n_2}$
we replace the subtree with root $v$ by $( \, \mu_{n_1}(t_1) \, \mu_{n_2}(t_2) \, )$.
For example,
\[
\adjustbox{valign=m}{
\begin{xy}
(  0,  0 )*{} = "root";
( -6, -6 )*{} = "l";
(  6, -6 )*{} = "r";
{ \ar@{-} "root"; "l" };
{ \ar@{-} "root"; "r" };
( -10, -12 )*{\ast} = "ll";
( -2, -12 )*{\ast} = "lr";
{ \ar@{-} "l"; "ll" };
{ \ar@{-} "l"; "lr" };
(  2, -12 )*{} = "rl";
( 10, -12 )*{\ast} = "rr";
{ \ar@{-} "r"; "rl" };
{ \ar@{-} "r"; "rr" };
( -2, -18 )*{\ast} = "rll";
(  6, -18 )*{} = "rlr";
{ \ar@{-} "rl"; "rll" };
{ \ar@{-} "rl"; "rlr" };
(  2, -24 )*{\ast} = "rlrl";
( 10, -24 )*{\ast} = "rlrr";
{ \ar@{-} "rlr"; "rlrl" };
{ \ar@{-} "rlr"; "rlrr" };
\end{xy}
}
\quad \xrightarrow{\qquad \scalebox{1}{$\mu_6$} \qquad} \quad
(\ast\ast)((\ast(\ast\ast))\ast)
\]
(We omit the outermost pair of parentheses corresponding to the root of the tree.)

If $t_1 \in T_{n_1}$ and $t_2 \in T_{n_2}$ then
the partial composition $t_1 \circ_i t_2 \in T_{n_1{+}n_2{-}1}$ for $1 \le i \le n_1$
is obtained by grafting the right tree into the left at position $i$:
that is, identifying leaf $i$ of $t_1$ (from left to right) with the root of $t_2$.
For example,
\[
\adjustbox{valign=m}{
\begin{xy}
(  0,  0 )*{} = "root";
( -6, -6 )*{} = "l";
(  6, -6 )*{} = "r";
{ \ar@{-} "root"; "l" };
{ \ar@{-} "root"; "r" };
( -10, -12 )*{\ast} = "ll";
( -2, -12 )*{\ast} = "lr";
{ \ar@{-} "l"; "ll" };
{ \ar@{-} "l"; "lr" };
(  2, -12 )*{\ast} = "rl";
( 10, -12 )*{\ast} = "rr";
{ \ar@{-} "r"; "rl" };
{ \ar@{-} "r"; "rr" };
\end{xy}
}
\quad \circ_3 \quad
\adjustbox{valign=m}{
\begin{xy}
(  0,  0 )*{} = "root";
( -4, -6 )*{\ast} = "l";
(  4, -6 )*{} = "r";
{ \ar@{-} "root"; "l" };
{ \ar@{-} "root"; "r" };
(  0, -12 )*{\ast} = "rl";
(  8, -12 )*{\ast} = "rr";
{ \ar@{-} "r"; "rl" };
{ \ar@{-} "r"; "rr" };
\end{xy}
}
\quad = \quad
\adjustbox{valign=m}{
\begin{xy}
(  0,  0 )*{} = "root";
( -6, -6 )*{} = "l";
(  6, -6 )*{} = "r";
{ \ar@{-} "root"; "l" };
{ \ar@{-} "root"; "r" };
( -10, -12 )*{\ast} = "ll";
( -2, -12 )*{\ast} = "lr";
{ \ar@{-} "l"; "ll" };
{ \ar@{-} "l"; "lr" };
(  2, -12 )*{} = "rl";
( 10, -12 )*{\ast} = "rr";
{ \ar@{-} "r"; "rl" };
{ \ar@{-} "r"; "rr" };
( -2, -18 )*{\ast} = "rll";
(  6, -18 )*{} = "rlr";
{ \ar@{-} "rl"; "rll" };
{ \ar@{-} "rl"; "rlr" };
(  2, -24 )*{\ast} = "rlrl";
( 10, -24 )*{\ast} = "rlrr";
{ \ar@{-} "rlr"; "rlrl" };
{ \ar@{-} "rlr"; "rlrr" };
\end{xy}
}
\]
In terms of association types, $\mu_{n_1}(t_1) \, \circ_i \, \mu_{n_2}(t_2)$ corresponds to
substitution of $( \, \mu_{n_2}(t_2) \, )$ for argument $i$ of $\mu_{n_1}(t_1)$;
we omit the parentheses if $n_2 = 1$.
For example,
$
(\ast\ast)(\ast\ast) \, \circ_3 \, \ast(\ast\ast)
\, = \,
(\ast\ast)((\ast(\ast\ast))\ast)
$.

Partial composition is nonassociative but satisfies
sequential and parallel axioms \cite[Definition 3.2.2.3]{BD} (see also \cite{MSS, LV}).
We state these axioms following \cite[Definition 11, Figure 1]{Markl}.
If $f \in T_m$, $g \in T_n$, $h \in T_p$ then
\[
( f \circ_j g ) \circ_i h
=
\begin{cases}
( f \circ_i h ) \circ_{j+p-1} g &\qquad 1 \le i \le j-1
\\
f \circ_j ( g \circ_{i-j+1} h ) &\qquad j \le i \le n+j-1
\\
( f \circ_{i-n+1} h ) \circ_j g &\qquad n+j \le i \le m+n-1
\end{cases}.
\]

Let $T$ denote the disjoint union of the $T_n$ for $n \ge 1$:
\[
T = \bigsqcup_{n \ge 1} T_n.
\]
The set $T$ together with all partial compositions is isomorphic to
the free nonsymmetric set operad generated by one binary operation
$\omega$ corresponding to the tree with root and two leaves.
(Nonsymmetric means that we have not yet introduced the action of
the symmetric group on the arguments.)
Let $T(n)$ denote the vector space with basis $T_n$.
On the direct sum
\[
\mathcal{T} = \bigoplus_{n \ge 1} \, T(n),
\]
we extend partial compositions so that they are linear in each factor.
The vector space $\mathcal{T}$ together with the extended partial compositions
is isomorphic to the free nonsymmetric vector operad generated by $\omega$.


\subsection{The free symmetric operad}

Consider an integer $n \ge 1$ and the set of indeterminates $\{ x_1, \dots, x_n \}$.
We write $S_n$ for the symmetric group of all $n!$ permutations of $\{ 1, \dots, n \}$.
For each $\alpha \in S_n$ and $t \in T_n$ we obtain the labelled tree $\alpha t$
consisting of $t$ with leaves labelled $\alpha(1), \dots, \alpha(n)$ from left to right.
We write $LT_n$ for the set of all such labelled trees.
Similarly, we apply the association type $\mu_n(t)$ for $t \in T_n$ to
the multilinear associative monomial $x_{\alpha(1)} \cdots x_{\alpha(n)}$
and obtain the nonassociative monomial $\alpha\mu_n(t)$.
We write $LU_n$ for the set of all such nonassociative monomials.
The bijection $\mu_n\colon T_n \to U_n$ extends in the obvious way to the bijection
$\lambda\mu_n\colon LT_n \to LU_n$.
For example,
\[
\adjustbox{valign=m}{
\begin{xy}
(  0,  0 )*{} = "root";
( -6, -6 )*{} = "l";
(  6, -6 )*{} = "r";
{ \ar@{-} "root"; "l" };
{ \ar@{-} "root"; "r" };
( -10, -12 )*+{6} = "ll";
( -2, -12 )*+{5} = "lr";
{ \ar@{-} "l"; "ll" };
{ \ar@{-} "l"; "lr" };
(  2, -12 )*{} = "rl";
( 10, -12 )*+{1} = "rr";
{ \ar@{-} "r"; "rl" };
{ \ar@{-} "r"; "rr" };
( -2, -18 )*+{4} = "rll";
(  6, -18 )*{} = "rlr";
{ \ar@{-} "rl"; "rll" };
{ \ar@{-} "rl"; "rlr" };
(  2, -24 )*+{2} = "rlrl";
( 10, -24 )*+{3} = "rlrr";
{ \ar@{-} "rlr"; "rlrl" };
{ \ar@{-} "rlr"; "rlrr" };
\end{xy}
}
\quad \xrightarrow{\quad \scalebox{1}{$\lambda\mu_6$} \quad} \quad
(x_6x_5)((x_4(x_2x_3))x_1)
\]

We extend partial compositions to labelled trees.
Consider two labelled trees $\alpha t \in LT_m$ and $\beta u \in LT_n$.
If $1 \le i \le m$ then the partial composition $\alpha t \circ_i \beta u \in LT_{m{+}n{-}1}$
must be a tree with $m{+}n{-}1$ leaves labelled by a permutation in $S_{m{+}n{-}1}$.
(Simple grafting of one labelled tree onto the other does not produce a permutation.)
This must be done in a manner which is equivariant with respect to the action of
the symmetric group.
Following \cite[Definion 1.37]{MSS} with minor changes, we have:
\begin{itemize}
\item
A leaf of $\alpha t$ with label $j$ for $1 \le j \le \alpha(i)-1$ retains its label.
\item
A leaf of $\beta u$ with label $j$ for $1 \le j \le n$ is relabelled $j+\alpha(i)-1$.
\item
A leaf of $\alpha t$ with label $j$ for $\alpha(i)+1 \le j \le m$ is relabelled $j+n-1$.
\end{itemize}
For example,
\[
\adjustbox{valign=m}{
\begin{xy}
(  0,  0 )*{} = "root";
( -6, -6 )*{} = "l";
(  6, -6 )*{} = "r";
{ \ar@{-} "root"; "l" };
{ \ar@{-} "root"; "r" };
( -10, -12 )*+{4} = "ll";
( -2, -12 )*+{3} = "lr";
{ \ar@{-} "l"; "ll" };
{ \ar@{-} "l"; "lr" };
(  2, -12 )*+{2} = "rl";
( 10, -12 )*+{1} = "rr";
{ \ar@{-} "r"; "rl" };
{ \ar@{-} "r"; "rr" };
\end{xy}
}
\quad \circ_3 \quad
\adjustbox{valign=m}{
\begin{xy}
(  0,  0 )*{} = "root";
( -4, -6 )*+{3} = "l";
(  4, -6 )*{} = "r";
{ \ar@{-} "root"; "l" };
{ \ar@{-} "root"; "r" };
(  0, -12 )*+{1} = "rl";
(  8, -12 )*+{2} = "rr";
{ \ar@{-} "r"; "rl" };
{ \ar@{-} "r"; "rr" };
\end{xy}
}
\quad = \quad
\adjustbox{valign=m}{
\begin{xy}
(  0,  0 )*{} = "root";
( -6, -6 )*{} = "l";
(  6, -6 )*{} = "r";
{ \ar@{-} "root"; "l" };
{ \ar@{-} "root"; "r" };
( -10, -12 )*+{6} = "ll";
( -2, -12 )*+{5} = "lr";
{ \ar@{-} "l"; "ll" };
{ \ar@{-} "l"; "lr" };
(  2, -12 )*{} = "rl";
( 10, -12 )*+{1} = "rr";
{ \ar@{-} "r"; "rl" };
{ \ar@{-} "r"; "rr" };
( -2, -18 )*+{4} = "rll";
(  6, -18 )*{} = "rlr";
{ \ar@{-} "rl"; "rll" };
{ \ar@{-} "rl"; "rlr" };
(  2, -24 )*+{2} = "rlrl";
( 10, -24 )*+{3} = "rlrr";
{ \ar@{-} "rlr"; "rlrl" };
{ \ar@{-} "rlr"; "rlrr" };
\end{xy}
}
\]

Let $LT(n)$ denote the $S_n$-module with linear basis $LT_n$;
we use the natural left action on labels (not on positions).
The direct sum of these $S_n$-modules,
\[
\mathcal{LT} = \bigoplus_{n \ge 1} LT(n),
\]
together with the bilinear extension of the partial compositions,
is isomorphic to the free symmetric vector operad generated by $\omega$.
(This binary operation has no symmetry: it is neither commutative nor anticommutative.)

An ideal $\mathcal{I}$ in the free symmetric operad $\mathcal{LT}$ is a
graded subspace (that is, $\mathcal{I}(n) \subseteq \mathcal{LT}(n)$ for $n \ge 1$)
such that
\begin{itemize}[leftmargin=*]
\item
$S_n \cdot \mathcal{I}(n) = \mathcal{I}(n)$:
each homogeneous space $\mathcal{I}(n)$ is an $S_n$-module
(that is, closed under the action of the symmetric group), and
\item
if $f \in \mathcal{I}(m)$ and $g \in \mathcal{LT}(n)$
then $f \circ_i g$ ($1 \le i \le m$) and $g \circ_j f$ ($1 \le j \le n$)
belong to $\mathcal{I}(m{+}n{-}1)$
(that is, $\mathcal{I}$ is closed under partial compositions).
\end{itemize}
The ideal $\langle f_1, f_2, \dots \rangle \subseteq \mathcal{LT}$ generated
by homogeneous elements $f_1, f_2, \dots$ is the smallest ideal
of $\mathcal{LT}$ containing $f_1, f_2, \dots$
If $\mathcal{I} = \langle f_1, f_2, \dots \rangle$ then we say that
$\mathcal G = \{ f_1, f_2, \dots \}$ is a minimal set of generators for $\mathcal{I}$ if
no proper subset of $\mathcal G$ generates $\mathcal{I}$;
this condition does not uniquely determine $\mathcal G$.


\subsection{Associativity, nullary operations, and the expansion map}

In general, an $n$-ary operation ($n \ge 0$) on a vector space $V$ is a multilinear map
$f\colon V^n \to V$.
For $n = 1$ we have $V^1 = V$ so a unary operation is simply a linear operator on $V$;
for $n = 0$ we have $V^0 = \mathbb{F}$ so a nullary operation is equivalent to the choice
of a constant vector $f(1) \in V$.
If we write $\mathrm{End}_n(V)$ for the vector space of all $n$-ary operations on $V$,
then the direct sum $\mathrm{End}(V) = \bigoplus_{n \ge 0} \mathrm{End}_n(V)$,
together with partial compositions (substitution of the output of one operation for
an input of another operation), is the \emph{endomorphism operad} of $V$.

Let $p$, $q$ be symbols denoting nullary operations on some underlying vector space.
For $n \ge 1$, consider monomials $v_1 v_2 \cdots v_{2n-1}$ with an odd number of factors
such that the $n$ odd-indexed factors $v_{2i-1}$ ($1 \le i \le n$) form
a multilinear associative monomial $x_{\alpha(1)} \cdots x_{\alpha(n)}$ for some
$\alpha \in S_n$, and each of the $n-1$ even-indexed factors is either $p$ or $q$.
We write $W_n$ for the set of all such monomials;
$S_n$ acts by permuting the odd-indexed factors.
For $v \in W_m$ and $w \in W_n$, we define $v \circ_i w \in W_{m+n-1}$ for $1 \le i \le m$
by substituting $w$ for $v_{2i-1}$ (with the appropriate change of labels).
We write $W(n)$ for the vector space whose basis consists of all such monomials.
The direct sum $\mathcal{W} = \bigoplus_{n \ge 1} W(n)$ is a suboperad of
the symmetric associative operad with two nullary operations.

\begin{definition}
The \emph{expansion map} $X_n \colon LT(n) \to W(n)$ on monomials
$\alpha t \in LT_n$ is defined recursively.
For a leaf with label $i$, we set $X_n( i ) = x_i$.
If $tu$ denotes an internal node with left and right subtrees
$t \in LT_r$ and $u \in LT_s$ with $r+s = n$ then
\[
X_n( tu ) = X_r(t) p X_s(u) -X_s(u) q X_r(t).
\]
That is, we replace each internal node by the operation $\ast_{pq}$.
\end{definition}

If we represent trees by nonassociative monomials and leaf labels by letters then
$X_2( ab ) = apb - bqa$ and
\begin{equation}
\label{X3}
\begin{array}{l}
X_3( (ab)c ) = ( apb {-} bqa ) p c - c q ( apb {-} bqa ) = apbpc - bqapc - cqapb + cqbqa,
\\[2pt]
X_3( a(bc) ) = a p ( bpc {-} cqb ) - ( bpc {-} cqb ) q a = apbpc - apcqb - bpcqa + cqbqa.
\end{array}
\end{equation}
For the expansions in degree 4, see Figure \ref{degree4expansions}.

\begin{figure}[ht]
$\boxed{
\begin{array}{r}
((ab)c)d
\;\longmapsto\;
  a p b p c p d
- d q a p b p c
- c q a p b p d
+ d q c q a p b \quad {}
\\
{}
- b q a p c p d
+ d q b q a p c
+ c q b q a p d
- d q c q b q a,
\\
(a(bc))d
\;\longmapsto\;
  a p b p c p d
- d q a p b p c
- b p c q a p d
+ d q b p c q a \quad {}
\\
{}
- a p c q b p d
+ d q a p c q b
+ c q b q a p d
- d q c q b q a,
\\
(ab)(cd)
\;\longmapsto\;
  a p b p c p d
- c p d q a p b
- a p b p d q c
+ d q c q a p b \quad {}
\\
{}
- b q a p c p d
+ c p d q b q a
+ b q a p d q c
- d q c q b q a,
\\
a((bc)d)
\;\longmapsto\;
  a p b p c p d
- b p c p d q a
- a p d q b p c
+ d q b p c q a \quad {}
\\
{}
- a p c q b p d
+ c q b p d q a
+ a p d q c q b
- d q c q b q a,
\\
a(b(cd))
\;\longmapsto\;
  a p b p c p d
- b p c p d q a
- a p c p d q b
+ c p d q b q a \quad {}
\\
{}
- a p b p d q c
+ b p d q c q a
+ a p d q c q b
- d q c q b q a.
\end{array}
}$
\vspace{-7pt}
\caption{Expansions of basis monomials in degree 4}
\label{degree4expansions}
\end{figure}

\begin{definition}
\label{kernel}
For each $n \ge 1$, the expansion map $X_n\colon \mathcal{LT}(n) \to \mathcal{W}(n)$
is a morphism of $S_n$-modules; we write $\mathcal{K}(n) = \ker(X_n)$.
Combining all the expansion maps we obtain the morphism of operads
$X \colon \mathcal{LT} \to \mathcal{W}$ with kernel
$\mathcal{K} = \bigoplus_{n \ge 1} \mathcal{K}(n)$.
The polynomial identities satisfied by $\ast_{pq}$
for all associative algebras $A$ and all $p, q \in A$
coincide with $\mathcal{K}$, which is an operad ideal in $\mathcal{LT}$.
These identities are the linear dependence relations
among the expansions of the nonassociative monomials.
We refer to $\mathcal{K}(n)$ as the $S_n$-module of \emph{all identities} in degree $n$.
\end{definition}

Our ultimate goal is to determine a set of generators for $\mathcal{K}$.


\section{Polynomial Identities in Degree $n \le 3$}

\begin{definition}
\label{lieadmissible}
In a nonassociative algebra, the \emph{Lie-admissible identity} is
\[
L(a,b,c) = \sum_{\sigma \in S_3} \varepsilon(\sigma) \big( (a^\sigma b^\sigma )c^\sigma  - a^\sigma (b^\sigma c^\sigma ) \big),
\]
where $\varepsilon\colon S_3 \to \{\pm1\}$ is the sign homomorphism.
If $L(a,b,c) \equiv 0$ then the commutator $xy - yx$ satisfies the Jacobi identity.
\end{definition}

We provide a different proof of the next result using elementary linear algebra.

\begin{theorem} \emph{\cite{M}.}
Over a field of characteristic 0,
every multilinear polynomial identity in degree $n \le 3$
satisfied by every mutation of every associative algebra is
a consequence of the Lie-admissible identity.
\end{theorem}

\begin{proof}
It is straightforward to verify that $\ker X_n = \{0\}$ for $1 \le n \le 2$.
The monomial basis of $\ff(3)$ consists of 12 elements ordered first by association type
and then by permutation of the variables:
\begin{equation}
\label{LT3basis}
\begin{array}{l}
(ab)c, \quad (ac)b, \quad (ba)c, \quad (bc)a, \quad (ca)b, \quad (cb)a,
\\
a(bc), \quad a(cb), \quad b(ac), \quad b(ca), \quad c(ab), \quad c(ba).
\end{array}
\end{equation}
The monomial basis of $\pq(3)$ consists of 24 elements ordered first by lex order
of the pair of nullary operations ($pp$, $pq$, $qp$, $qq$) and then by permutation of the variables:
\begin{equation}
\label{W3basis}
\begin{array}{l@{\quad}l@{\quad}l@{\quad}l@{\quad}l@{\quad}l}
apbpc, & apcpb, & bpapc, & bpcpa, & cpapb, & cpbpa,
\\
apbqc, & apcqb, & bpaqc, & bpcqa, & cpaqb, & cpbqa,
\\
aqbpc, & aqcpb, & bqapc, & bqcpa, & cqapb, & cqbpa,
\\
aqbqc, & aqcqb, & bqaqc, & bqcqa, & cqaqb, & cqbqa.
\end{array}
\end{equation}
The expansion map $X_3\colon \ff(3) \to \pq(3)$ is determined by its values on
the nonassociative monomials with the identity permutation of the arguments;
see \eqref{X3}.
We apply all permutations in $S_3$ to the arguments $a, b, c$
and store the coefficients of the monomials in the $24 \times 12$ matrix $E_3$
representing $X_3$ with respect to the ordered bases; see Figure \ref{E3matrix}.
That is, the $(i,j)$ entry of $E_3$ is the coefficient of the $i$th associative monomial
\eqref{W3basis} in the expansion of the $j$th nonassociative monomial \eqref{LT3basis}.
It is easy to check that this matrix has rank 11 and hence nullity 1,
and that a basis for its nullspace is the coefficient vector of the Lie-admissible identity.
\end{proof}

\begin{figure}[ht]
\small
$
\left[
\begin{array}{rrrrrrrrrrrr}
 1 &  \cdot &  \cdot &  \cdot &  \cdot &  \cdot &  1 &  \cdot &  \cdot &  \cdot &  \cdot &  \cdot \\[-2pt]
 \cdot &  1 &  \cdot &  \cdot &  \cdot &  \cdot &  \cdot &  1 &  \cdot &  \cdot &  \cdot &  \cdot \\[-2pt]
 \cdot &  \cdot &  1 &  \cdot &  \cdot &  \cdot &  \cdot &  \cdot &  1 &  \cdot &  \cdot &  \cdot \\[-2pt]
 \cdot &  \cdot &  \cdot &  1 &  \cdot &  \cdot &  \cdot &  \cdot &  \cdot &  1 &  \cdot &  \cdot \\[-2pt]
 \cdot &  \cdot &  \cdot &  \cdot &  1 &  \cdot &  \cdot &  \cdot &  \cdot &  \cdot &  1 &  \cdot \\[-2pt]
 \cdot &  \cdot &  \cdot &  \cdot &  \cdot &  1 &  \cdot &  \cdot &  \cdot &  \cdot &  \cdot &  1 \\[-2pt]
 \cdot &  \cdot &  \cdot &  \cdot &  \cdot &  \cdot &  \cdot & -1 &  \cdot &  \cdot & -1 &  \cdot \\[-2pt]
 \cdot &  \cdot &  \cdot &  \cdot &  \cdot &  \cdot & -1 &  \cdot & -1 &  \cdot &  \cdot &  \cdot \\[-2pt]
 \cdot &  \cdot &  \cdot &  \cdot &  \cdot &  \cdot &  \cdot &  \cdot &  \cdot & -1 &  \cdot & -1 \\[-2pt]
 \cdot &  \cdot &  \cdot &  \cdot &  \cdot &  \cdot & -1 &  \cdot & -1 &  \cdot &  \cdot &  \cdot \\[-2pt]
 \cdot &  \cdot &  \cdot &  \cdot &  \cdot &  \cdot &  \cdot &  \cdot &  \cdot & -1 &  \cdot & -1 \\[-2pt]
 \cdot &  \cdot &  \cdot &  \cdot &  \cdot &  \cdot &  \cdot & -1 &  \cdot &  \cdot & -1 &  \cdot \\[-2pt]
 \cdot &  \cdot & -1 & -1 &  \cdot &  \cdot &  \cdot &  \cdot &  \cdot &  \cdot &  \cdot &  \cdot \\[-2pt]
 \cdot &  \cdot &  \cdot &  \cdot & -1 & -1 &  \cdot &  \cdot &  \cdot &  \cdot &  \cdot &  \cdot \\[-2pt]
-1 & -1 &  \cdot &  \cdot &  \cdot &  \cdot &  \cdot &  \cdot &  \cdot &  \cdot &  \cdot &  \cdot \\[-2pt]
 \cdot &  \cdot &  \cdot &  \cdot & -1 & -1 &  \cdot &  \cdot &  \cdot &  \cdot &  \cdot &  \cdot \\[-2pt]
-1 & -1 &  \cdot &  \cdot &  \cdot &  \cdot &  \cdot &  \cdot &  \cdot &  \cdot &  \cdot &  \cdot \\[-2pt]
 \cdot &  \cdot & -1 & -1 &  \cdot &  \cdot &  \cdot &  \cdot &  \cdot &  \cdot &  \cdot &  \cdot \\[-2pt]
 \cdot &  \cdot &  \cdot &  \cdot &  \cdot &  1 &  \cdot &  \cdot &  \cdot &  \cdot &  \cdot &  1 \\[-2pt]
 \cdot &  \cdot &  \cdot &  1 &  \cdot &  \cdot &  \cdot &  \cdot &  \cdot &  1 &  \cdot &  \cdot \\[-2pt]
 \cdot &  \cdot &  \cdot &  \cdot &  1 &  \cdot &  \cdot &  \cdot &  \cdot &  \cdot &  1 &  \cdot \\[-2pt]
 \cdot &  1 &  \cdot &  \cdot &  \cdot &  \cdot &  \cdot &  1 &  \cdot &  \cdot &  \cdot &  \cdot \\[-2pt]
 \cdot &  \cdot &  1 &  \cdot &  \cdot &  \cdot &  \cdot &  \cdot &  1 &  \cdot &  \cdot &  \cdot \\[-2pt]
 1 &  \cdot &  \cdot &  \cdot &  \cdot &  \cdot &  1 &  \cdot &  \cdot &  \cdot &  \cdot &  \cdot
\end{array} \right]
$
\vspace{-2mm}
\caption{The matrix $E_3$ representing the expansion map $X_3$ (here a dot represents a zero entry)}
\label{E3matrix}
\end{figure}


\section{Polynomial Identities in Degree $4$}

Montaner \cite{M} (see also \cite[Chapter 5]{EMbook}) showed that every identity
in degree $n \le 4$ satisfied by every mutation algebra is a consequence of
the Lie-admissible identity, the Jordan-admissible identity, and two further
identities; furthermore, none of these identities is a consequence of the other three.
In this section we use computer algebra to simplify this result:
we discover two new multilinear identities in degree 4,
which are not consequences of the Lie-admissible identity,
and which generate all identities in degree 4 (including the Jordan-admissible identity).

\begin{definition}
\label{defJ}
In a nonassociative algebra, the \emph{linearized Jordan identity} is
\[
  ( ( b c ) a ) d
+ ( ( b d ) a ) c
+ ( ( c d ) a ) b
- ( a b ) ( c d )
- ( a c ) ( b d )
- ( a d ) ( b c ).
\]
If we expand each nonassociative product $xy$ as the anticommutator $xy + yx$
then we obtain the \emph{Jordan-admissible identity}:
\[
\begin{array}{l}
J(a,b,c,d) =
  ((bc)a)d
+ ((bd)a)c
+ ((cb)a)d
+ ((cd)a)b
+ ((db)a)c
+ ((dc)a)b
\\ {}
+ (a(bc))d
+ (a(bd))c
+ (a(cb))d
+ (a(cd))b
+ (a(db))c
+ (a(dc))b
- (ab)(cd)
\\ {}
- (ab)(dc)
- (ac)(bd)
- (ac)(db)
- (ad)(bc)
- (ad)(cb)
- (ba)(cd)
- (ba)(dc)
\\ {}
- (bc)(ad)
- (bc)(da)
- (bd)(ac)
- (bd)(ca)
- (ca)(bd)
- (ca)(db)
- (cb)(ad)
\\ {}
- (cb)(da)
- (cd)(ab)
- (cd)(ba)
- (da)(bc)
- (da)(cb)
- (db)(ac)
- (db)(ca)
\\ {}
- (dc)(ab)
- (dc)(ba)
+ b((cd)a)
+ b((dc)a)
+ c((bd)a)
+ c((db)a)
+ d((bc)a)
\\ {}
+ d((cb)a)
+ b(a(cd))
+ b(a(dc))
+ c(a(bd))
+ c(a(db))
+ d(a(bc))
+ d(a(cb)).
\end{array}
\]
If $J(a,b,c,d) \equiv 0$ then the anticommutator $xy + yx$ satisfies the Jordan identity.
\end{definition}

\begin{definition}
\label{defHI}
In a nonassociative algebra, we consider the following identities
where $(x,y,z) = (xy)z - x(yz)$ and $x \circ y = xy + yx$:
\begin{align*}
H(a,b,c,d)
&=
\big(
(a,c,b) + (b,a,c) + (c,b,a)
\big)d
-
\!\!
\sum_{\sigma \in S_3}
\big(
( a^\sigma b^\sigma ) ( c^\sigma d )
-
a^\sigma ( ( b^\sigma c^\sigma ) d )
\big),
\\
I(a,b,c,d)
&=
(bc,a,d) - (a,bc,d) + (a,d,bc)
+ ( b, a \circ d, c ) - (b,d,c) \circ a
\\
& - (b,a,c) \circ d.
\end{align*}
In identity $H$, the first three terms include a cyclic sum of associators
\cite[Equation (5)]{KKOR},
and each term in the summation can be written as the difference of two associators.
\end{definition}

The next result improves \cite[Theorem 2.3]{M}:
we have only two new identities in degree 4, not three.
In addition, even though our new identities $H$ and $I$ are multilinear,
each contains only 18 nonassociative monomials
(after expanding the associators and anticommutators),
whereas the new identities of \cite{M} have
48 monomials (the Jordan-admissible identity),
20 monomials (identity $E$),
and 52 monomials (identity $K$). Furthermore, our new identities have only coefficients $\pm 1$,
whereas identity $K$ has coefficients 1, 2, 4, 6. Finally, we show that the new identities also generate the consequences of the Lie-admissible identity, and thus the latter do not need to be considered.

\begin{theorem} \label{degree4}
Every identity in degree 4 satisfied by every mutation algebra follows from
the identities $H$ and $I$ from Definition \ref{defHI}.
\end{theorem}

\begin{proof}
We first consider the expansion matrix.
The monomial basis of $\ff(4)$ consists of 120 elements ordered first by association type
and then by lex order of permutations $\sigma \in S_4$ (indicated by superscripts):
\begin{equation}
\label{LT4basis}
( ( a^\sigma b^\sigma ) c^\sigma ) d^\sigma, \;\;
( a^\sigma ( b^\sigma c^\sigma ) ) d^\sigma, \;\;
( a^\sigma b^\sigma ) ( c^\sigma d^\sigma ), \;\;
a^\sigma ( ( b^\sigma c^\sigma ) d^\sigma ), \;\;
a^\sigma ( b^\sigma ( c^\sigma d^\sigma ) ).
\end{equation}
The monomial basis of $\pq(4)$ consists of 192 elements ordered first by lex order of
the triple of nullary operations and then by lex order of permutations $\sigma \in S_4$:
\begin{equation}
\label{W4basis}
\begin{array}{l@{\quad}l@{\quad}l@{\quad}l}
a^\sigma p \, b^\sigma p \, c^\sigma p \, d^\sigma, &
a^\sigma p \, b^\sigma p \, c^\sigma q \, d^\sigma, &
a^\sigma p \, b^\sigma q \, c^\sigma p \, d^\sigma, &
a^\sigma p \, b^\sigma q \, c^\sigma q \, d^\sigma,
\\
a^\sigma q \, b^\sigma p \, c^\sigma p \, d^\sigma, &
a^\sigma q \, b^\sigma p \, c^\sigma q \, d^\sigma, &
a^\sigma q \, b^\sigma q \, c^\sigma p \, d^\sigma, &
a^\sigma q \, b^\sigma q \, c^\sigma q \, d^\sigma.
\end{array}
\end{equation}
The expansion map $X_4\colon \ff(4) \to \pq(4)$ is determined by its values on
the nonassociative monomials with the identity permutation of the arguments
(Figure \ref{degree4expansions}).
We apply all permutations in $S_4$ to the arguments $a, b, c, d$ in the expansions
and store the coefficients in the $192 \times 120$ matrix $E_4$ representing $X_4$
with respect to the ordered bases \eqref{LT4basis} and \eqref{W4basis}.
The $(i, j)$ entry of $E_4$ is the coefficient of the $i$th associative
monomial in the expansion of the $j$th nonassociative monomial.
Thus each column of $E_4$ contains 1 and $-1$ each four times.

Next, we consider the consequences of the Lie-admissible identity.
The identity $L(a,b,c) \in \mathcal{LT}(3)$ is skew-symmetric:
\[
L( a^\sigma, b^\sigma, c^\sigma ) = \epsilon( \sigma ) L(a,b,c).
\]
We write $\mathcal{L} \subset \mathcal{LT}$ for the operad ideal generated by $L$;
clearly $\mathcal{L} \subseteq \mathcal{K}$.
The homogeneous component $\mathcal{L}(4)$ is generated as an $S_4$-module by
the partial compositions
\begin{equation}
\label{Lie4generators}
\begin{array}{l}
L \circ_1 \omega = L( \omega(a,b), c, d ) = L(ab,c,d),
\\
\omega \circ_1 L = \omega( L(a,b,c), d ) = L(a,b,c)d,
\\
\omega \circ_2 L = \omega( a, L(b,c,d) ) = aL(b,c,d).
\end{array}
\end{equation}
We refer to the elements of $\mathcal{L}(4)$ as the \emph{old identities} in degree 4.
Applying all permutations $\sigma \in S_4$ to the generators \eqref{Lie4generators}
allows us to represent $\mathcal{L}(4)$ as the row space of the $72 \times 120$ matrix $C_4$ whose columns are labelled by the monomials \eqref{LT4basis}.
The row space of $C_4$ is a subspace (in fact an $S_4$-submodule) of
the nullspace of the matrix $E_4$. We set $o_4 = \mathrm{rank}( C_4)$ and write $\overline{C}_4$ for the $o_4 \times 120$
matrix in RCF (row canonical form) whose row space equals that of $C_4$.

Finally, we consider the new identities.
The elements of the nullspace of $E_4$ are the coefficient vectors of $\mathcal{K}(4)$.
We set $a_4 = \mathrm{nullity}(E_4)$ and write $N_4$ for the $a_4 \times 120$ matrix
in RCF whose row space is the nullspace of $E_4$.
The rows of $N_4$ span the $S_4$-module of all identities in degree 4.
Clearly the row space of $\overline{C}_4$ is a subspace of the row space of $N_4$,
and hence $o_4 \le a_4$.
The quotient $\mathcal{K}(4) / \mathcal{L}(4)$ is the $S_4$-module of \emph{new identities}
in degree 4, and its dimension is $a_4 - o_4$.

Let $A, O \subseteq \{1,\dots,120\}$ be the column indices of the leading 1s in $N_4$ and
$\overline{C}_4$ respectively.
A linear basis of $\mathcal{K}(4) / \mathcal{L}(4)$ corresponds to (the cosets of) the rows
of $N_4$ whose leading 1s have column indices in $A \setminus O$.
It is straightforward using the module generators algorithm \cite{BMP}
to compute a subset of this linear basis which represents a set of
$S_4$-module generators for the quotient module.
Computations with the computer algebra system SageMath
show that $a_4 = 88$ and hence $N_4$ has rank $n_4 = 32$;
the nonzero entries of $N_4$ are $\pm \frac12, \pm 1, -\frac32$.
For each row of $N_4$, we multiply the coefficients by the LCM of their denominators
to obtain integers and then divide by the GCD of these integers to obtain vectors with relatively prime integer coefficients.
The squared Euclidean lengths of the resulting vectors with multiplicities in parentheses are
\[
12 \, (4), \; 18 \, (4), \; 42 \, (8), \; 48 \, (2), \; 56 \, (1), \; 60 \, (4), \; 64 \, (1), \; 72 \, (2), \; 74 \, (3), \; 82 \, (1), \; 100 \, (2).
\]
We sort the rows of the new integer matrix, also called $N_4$, by increasing length.
Further SageMath computations show that $o_4 = 19$, which implies that
the quotient module $\mathcal{K}(4)/\mathcal{L}(4)$ has dimension 13.

We next use the module generators algorithm again to
determine the smallest subset of the shortest rows of $N_4$ which generates
the quotient module $\mathcal{K}(4) / \mathcal{L}(4)$.
We obtain two identities and verify that neither is a consequence of the other.
The first has 18 terms and coefficients $\pm 1$ (squared length 18);
the second has 33 terms and coefficients $\pm 1, \pm 2$ (squared length 42).

We can obtain better results using linear algebra over the integers;
this requires replacing the RCF by the HNF (Hermite normal form),
and applying the LLL algorithm \cite{BP} to determine shorter integer vectors.

The entries of the matrix $E_4$ belong to $\{ 0, \pm 1 \}$.
We compute the HNF of the transpose $E_4^t$, denoted by $V$, and
a square matrix $U$ with $\det(U) = \pm 1$ such that $U E_4^t = V$.
Since $E_4^t$ has rank 88, the bottom 32 rows of $V$ are zero,
and hence the bottom 32 rows of $U$ form a matrix $N$
whose rows form a lattice basis of the left nullspace of $E_4^t$
(the right nullspace of $E_4$).
(By a lattice basis we mean a set of free generators for a submodule of a free $\mathbb{Z}$-module.)

After applying the LLL algorithm to the lattice generated by the rows of $N$,
we obtain a matrix $N_{LLL}$ whose nonzero entries are $\pm 1$ and whose rows have
the following squared Euclidean lengths with multiplicities given in parentheses:
\[
12 \, (13), \; 18 \, (12), \; 24 \, (1), \; 26 \, (1), \; 28 \, (1), \; 32 \, (3), \; 34 \, (1).
\]
Further computations show that the quotient module $\mathcal{K}(4)/\mathcal{L}(4)$
is generated by two rows of $N_{LLL}$ with squared length 18.
These are the coefficient vectors of the identities
$I(a,b,c,d)$ and $H(a,b,c,d)$.
Moreover, the $S_4$-module generated by $I(a,b,c,d)$ and $H(a,b,c,d)$ has dimension 32, so it coincides with $\mathcal K(4)$.
\end{proof}

\begin{remark}
Since the dimension of $\mathcal K(4)$ is 32 and the permutation of variables of a multilinear identity of degree $4$ can produce at most $4!=24$ linearly independent identities, a lower bound on the number of generators of $\mathcal K(4)$ as an $S_4$-module is $\lceil 32/24\rceil = 2$. Therefore the set of generators $\{H,I\}$ of $\mathcal K(4)$ has minimum cardinal.
\end{remark}

\begin{corollary}
\label{corjor}
Consider these three consequences of the Lie-admissible identity:
\[
P(a,b,c,d) = L(ab,c,d), \quad
Q(a,b,c,d) = L(a,b,c)d, \quad
R(a,b,c,d) = aL(b,c,d).
\]
Then
\begin{align*}
P(a,b,c,d) &= I(c,a,b,d)-I(d,a,b,c),
\\
Q(a,b,c,d) &= H(a,c,b,d)-H(a,b,c,d),
\\
2R(a,b,c,d) &= \sum_{\sigma\in S_3}\varepsilon(\sigma)\big(H(a, b^\sigma, c^\sigma, d^\sigma)+I(a, b^\sigma, c^\sigma, d^\sigma) + I(c^\sigma, a ,b^\sigma,d^\sigma)\big)
\\
& \, + \sum_{\sigma\in S_2} \varepsilon(\sigma)H(b,c^\sigma,d^\sigma,a),
\\
2J(a,b,c,d) &= \sum_{\sigma\in S_3}\big(H(a, b^\sigma, c^\sigma, d^\sigma)+I(a, b^\sigma, c^\sigma, d^\sigma)+I(c^\sigma,a,b^\sigma,d^\sigma)\big)
\\
& \, +\sum_{\sigma\in S_2} H(b,c^\sigma,d^\sigma,a),
\end{align*}
where $J(a,b,c,d)$ stands for the Jordan-admissible identity.
\end{corollary}

\begin{proof}
Straightforward computation.
\end{proof}


\section{Polynomial Identities in Degree 5}

In degree 5 there are 14 association types and
hence $14 \cdot 5! = 1680$ multilinear nonassociative monomials;
there are $5! \cdot 2^4 = 1920$ associative $pq$-monomials.

Recall that in degree 4, identities $H$ and $I$ from Definition \ref{defHI} generate the kernel $\mathcal K(4)$ of the expansion map
as an $S_4$-module.
Each identity $U$ in degree 4 produces six consequences in degree 5:
\[
U(ab,c,d,e), \;\,
U(a,bc,d,e), \;\,
U(a,b,cd,e), \;\,
U(a,b,c,de), \;\,
U(a,b,c,d)e, \;\,
aU(b,c,d,e).
\]



\begin{theorem}\label{K(5)generators}
Every identity in degree 5 satisfied by every mutation algebra follows from the consequences
of $H$ and $I$ in degree 4, and the new identity $G$ in degree 5 displayed in Figure \ref{new5}.
\end{theorem}

\begin{proof}
The proof is similar to that of degree $4$. We order the monomial bases of $\mathcal{LT}(5)$ and $\mathcal W(5)$ as in Theorem \ref{degree4}. We need to perform computations on the $1920 \times 1680$ matrix $E_5$ representing the expansion map $X_5:\mathcal{LT}(5)\rightarrow\mathcal W(5)$ (with respect to the monomial bases above). To this end, we use the class of rational sparse matrices in SageMath.

The kernel $\mathcal K(5)$ of the expansion map is an $S_5$-module of dimension 778 (comprising all identities). The twelve consequences (in degree 5) of identities $H$ and $I$ generate the $S_5$-module $\mathcal O(5)$ of old identities, which has dimension 747.
Hence the quotient module $\mathcal K(5)/\mathcal O(5)$ of new identities has dimension 31. We compute the HNF, denoted by $V$, of $E_5^t$ and a square matrix $U$ with $\det(U) = \pm1$ such that $UE_5^t = V$. The bottom 778 rows of $U$ produce a matrix $N$ whose rows form a lattice basis of the right nullspace of $E_5$. Next, we apply the LLL algorithm to the lattice generated by the rows of $N$ to obtain the matrix $N_{LLL}$; we find that the $S_5$-module $\mathcal K(5)/\mathcal O(5)$ is generated by one row of $N_{LLL}$ having 48 nonzero $\pm 1$ entries, which is the coefficient vector of identity $G(a,b,c,d,e)$. The computations required around 4GB of RAM, and had a runtime of 90 minutes, in an AMD Ryzen 5 5600X processor at 3.70GHz running SageMath 9.2 on Windows 10.
\end{proof}

\begin{figure}[ht]
\small
\begin{empheq}[box=\fbox]{align*}
& G(a,b,c,d,e) = \sum_{\substack{\sigma \in S_2(\{a,c\}) \\ \tau\in S_2(\{d,e\})}} \varepsilon(\sigma)\bigg(\{e,d,a^\sigma(c^\sigma b)\}- (\{e,d,b\}a^\sigma)c^\sigma + (e,b,a^\sigma(c^\sigma d))
\\
&-((ea^\sigma)c^\sigma,b,d) +(e,d,(ba^\sigma)c^\sigma) -(e,(ba^\sigma)c^\sigma,d) +((e,b,d)a^\sigma)c^\sigma-a^\sigma(c^\sigma(e,d,b)) \\
&+((d^\tau a^\sigma)c^\sigma)(e^\tau b) -d^\tau(a^\sigma(c^\sigma(e^\tau b))) +d^\tau(((e^\tau b)a^\sigma)c^\sigma)  -d^\tau(((e^\tau a^\sigma)c^\sigma)b)\bigg),\\
&\text{where }\{x,y,z\}=(xy)z + y(xz),\text{ and }S_2(\{x,y\})\text{ denotes }S_2\text{ acting on the set }\{x,y\}.
\end{empheq}
\vspace{-3mm}
\caption{The new identity in degree 5}
\label{new5}
\end{figure}

\begin{remark}
The dimension of $\mathcal K(5)$ is 778 and permuting the variables of a multilinear identity of degree $5$ can produce at most $5!=120$ linearly independent identities, so a lower bound on the number of generators of $\mathcal K(5)$ as an $S_5$-module is $\lceil 778/120\rceil = 7$. In Theorem \ref{K(5)generators} we have obtained a set with 13 generators: the 12 consequences of identities $H$ and $I$ plus a new identity $G$. In fact, it can be checked that $\mathcal K(5)$ is already generated by identity $G$, the consequences of identity $I$, and consequences $H(ab,c,d,e)$, $H(a,b,c,d)e$ and $aH(b,c,d,e)$ of identity $H$. Therefore an upper bound on the minimum number of generators of $\mathcal K(5)$ as an $S_5$-module is $10$.
\end{remark}


\section{Polynomial Identities in Degree 6}

In degree 6 there are 42 association types and hence $42\cdot 6!=30240$ multilinear nonassociative monomials, and there are $6!\cdot 2^5=23040$ associative $pq$-monomials. So, to represent the expansion map $X_6$ as a whole we would need to use a matrix of size $30240 \times 23040$, which is too large to manipulate with our computer system. We use the representation theory of the symmetric group
to reduce the problem to a set of matrices of smaller sizes and
demonstrate the existence of a number of new identities in degree 6.
We choose a set of conjugacy class representatives in $S_6$ and
calculate the matrices representing these permutations on the
modules of old and all identities. Comparing the traces of these matrices with the character table of $S_6$, we obtain the multiplicities of the irreducible representations of the $S_6$-modules.

\begin{theorem}
For each of the 11 partitions $\lambda$ of 6, the following table contains
the multiplicity of each irreducible representation in the $S_6$-modules
of all identities (the kernel of the expansion map), the old identities
(the consequences of the identities of lower degree), and the quotient
module of new identities (the difference of the previous two multiplicities):
\[
\begin{array}{r|ccccccccccc}
\lambda \;\; & 6 & 51 & 42 & 411 & 33 & 321 & 3111 & 222 & 2211 & 21111 & 111111
\\
\dim(\lambda) & 1 & 5 & 9 & 10 & 5 & 16 & 10 & 5 & 9 & 5 & 1
\\ \midrule
\emph{all}(\lambda) & 41 & 205 & 369 & 410 & 205 & 656 & 410 & 205 & 369 & 205 & 41
\\
\emph{old}(\lambda) & 29 & 136 & 237 & 268 & 131 & 422 & 267 & 131 & 236 & 133 & 28
\\
\emph{new}(\lambda) & 12 & 69 & 132 & 142 & 74 & 234 & 143 & 74 & 133 & 72 & 13
\end{array}
\]
Furthermore, the dimension of the quotient module of new identities is
\[
\sum_\lambda \mathrm{new}(\lambda) \, \mathrm{dim}(\lambda) = 10449.
\]
\end{theorem}

\begin{proof}
These methods have been described in detail in \cite[\S\S 2.4--2.7]{BMP}.
\end{proof}

\begin{conjecture}
The kernel of the expansion map in all degrees, that is,
the operad ideal $\mathcal{K} = \bigoplus_{n \ge 1} \mathcal{K}(n)$
(see Definition \ref{kernel}),
is not finitely generated.
In other words, no finite set of identities generates all the identities satisfied
by all mutation algebras.
\end{conjecture}



\end{document}